\documentclass[10pt,a4paper]{article}

\usepackage[T1]{fontenc} 

\usepackage[english]{babel}

\usepackage{amsmath}
\usepackage{amsfonts}
\usepackage{amssymb}
\usepackage{dsfont} 
\usepackage{amsthm} 
\usepackage{esint} 

\usepackage{csquotes} 
\usepackage{xcolor} 
\usepackage{enumitem} 
\usepackage{hyperref} 
\usepackage{verbatim} 

\usepackage[capitalise]{cleveref}
\usepackage[backend=bibtex, citestyle=alphabetic, bibstyle=alphabetic, isbn=false, maxnames=5, url = false]{biblatex} 

\DeclareSourcemap{
  \maps[datatype=bibtex]{
    \map{
      \step[fieldsource=doi,final]
      \step[fieldset=url,null]
      \step[fieldset=urldate,null]
    }
  }
}

\newcommand\R{\mathbb{R}}

\newcommand\N{\mathbb{N}}

\newcommand{\fsL}{\textnormal{L}} 
\newcommand{\fsH}{\textnormal{H}} 
\newcommand{\fsW}{\textnormal{W}} 
\newcommand{\Lip}{\mathrm{Lip}} 
\newcommand{\init}{\mathrm{in}} 
\newcommand{\fsC}{\mathcal{C}} 

\newtheorem{theo}{Theorem}

\newtheorem{lemm}{Lemma}

\newtheorem{prop}{Proposition}
\newtheorem{defi}{Definition}
\newtheorem{rema}{Remark}

\crefname{theo}{Theorem}{Theorems}
\crefname{prop}{Proposition}{Propositions}
\crefname{lemm}{Lemma}{Lemmas}

\newcommand\eps{\varepsilon}
\renewcommand{\phi}{\varphi}

\newcommand*{\dd}{\mathop{}\!\mathrm{d}}
\newcommand\indic[1]{\mathds{1}_{#1}}

\author{Hector Bouton
  \thanks{Email: \href{bouton@imj-prg.fr}
    {\texttt{bouton@imj-prg.fr}}\\
  Université Paris Cité and Sorbonne Université, CNRS, IMJ-PRG, F-75013 Paris, France}}
\title{Absence of gelation and regularity for coagulation-fragmentation-diffusion equations with critical coagulation rate in low dimension}
\date{} 

\begin{document}

\maketitle
\begin{abstract}
  We consider a discrete coagulation-fragmentation-diffusion system on a bounded domain of $\R^d$. In dimension $d\leq 2$, we prove mass conservation (that is absence of gelation) when the coagulation coefficient satisfies the critical additive bound $a_{i, j} \leq C(i + j)$ for some constant $C > 0$. We thereby close the remaining gap between the spatially homogeneous setting and the spatially inhomogeneous diffusive setting.

  In dimension $d\leq 4$, we impose a higher decay rate on the coagulation and fragmentation coefficients, but we nonetheless prove regularity and absence of gelation under mild assumptions on the diffusion rates. 
\end{abstract}

\section{Introduction}
\subsection{Background and motivation}
We consider in this paper a discrete coagulation-fragmentation-diffusion equation on a bounded domain $\Omega$ of $\R^d$ with homogeneous Neumann boundary conditions. This system of infinitely many equations describes the evolution of clusters $c_i = c_i(t, x)$ that coagulate, fragment and diffuse. For $i \in \N^* := \N\setminus\{0\}$, we consider
\begin{equation}\label{eq:coagulation}
  \left\{
  \begin{aligned}
  & \partial_t c_i - d_i \Delta c_i = Q_i(c) + F_i(c)\quad
  &&\text{in } [0, T]\times\Omega, \\
  & \nabla c_i\cdot n = 0 \quad
  &&\text{on }  [0,T]\times\partial\Omega, \\
  & c_i(0,\cdot) = c_{i, in} \quad 
  &&\text{in } \Omega,
\end{aligned}
\right.
\end{equation}
where $T > 0$ and $\Omega$ is a smooth bounded domain of $\R^d$. The coagulation operator $Q_i$ is given by:
\begin{equation} \label{eq:Q}
  Q_i(c) := Q_i^+(c) - Q_i^-(c) = \frac{1}{2}\sum_{j=1}^{i-1}a_{i-j,j}c_{i-j}c_j - \sum_{j=1}^{+\infty}a_{i,j}c_ic_j,
\end{equation}
for some non-negative coefficients $a_{i, j}$.

The fragmentation operator takes the form
\begin{equation} \label{eq:F}
  F_i(c) := F_i^+(c) - F_i^-(c) = \sum_{j=1}^{+\infty}B_{i+j}\beta_{i + j, i}c_{i+j} - B_ic_i,
\end{equation}
for non-negative coefficients $B_i$ and $\beta_{i, j}$.\medskip

We make the following assumptions on the coefficients of the problem:
\begin{enumerate}[label={(A\arabic*)}]
  \item\label{ass:symmetry} \textit{Symmetry of the coagulation rate:}
  $$\forall i, j \in \N^*, \, a_{i, j} = a_{j, i}.$$
  \item\label{ass:bounda}\textit{Linear bound on the coagulation rate:}
  $$\exists C_a > 0, \, \forall i, j \in \N^*, \, 0 \leq a_{i, j} \leq C_a(i + j).$$
  \item\label{ass:boundB} \textit{At most polynomial growth of the fragmentation rate:}
  $$\exists C_B > 0, \, m \in \N, \, \forall i\in \N^*, \, 0 \leq B_{i} \leq C_B i^m.$$
  \item\label{ass:massB}\textit{Formal conservation of mass during fragmentation:}
  $$\forall i\in \N\setminus\{0, 1\}, \, i = \sum_{j = 1}^{i - 1} j\beta_{i, j}\quad\text{and }\forall i, j\in \N\setminus\{0\}, \, \beta_{i, j} \geq 0.$$
  \item\label{ass:boundd}\textit{Upper and lower boundedness of the diffusion rate:}
  $$\exists \underline D,\, \overline D > 0, \, \forall i \in \N^*, \, \underline D \leq d_i \leq \overline D.$$
\end{enumerate}

We call parameters of the problem the constants $C_a$, $C_B$, $m$, $\underline D$ and $\overline D$.

This equation models clusters of size $i$ and $j$ that coagulate at rate $a_{i, j}$ to form a cluster of size $i + j$. In return, clusters of size $i$ break up at rate $B_i$ to form an average number $\beta_{i, j}$ of clusters of size $j$.\medskip 

This type of coagulation-fragmentation equations has numerous applications in aerosol dynamics, animal grouping, formation of polymers, description of raindrops, formation of galaxies, study of hematology and many others. We refer to the surveys \cite{drake1972, laurencotmischler2004, banasiaklamb2019} and references therein. The homogeneous case without diffusion has been widely studied and the behavior of solutions is precisely described (see for instance the books \cite{banasiaklamb2019, banasiaklamb2019a} and references therein). Much less is known about the coagulation-fragmentation-diffusion system. \medskip

Under the assumption \ref{ass:massB}, one can check that the total mass $\int_\Omega \rho_1 := \int_\Omega \sum_{i = 1}^{+\infty} i c_i$ is formally conserved (cf.~formulas \eqref{eq:weak} and \eqref{eq:weakF} below with $\phi_i = i$). However, this formal result may fail, and it may happen that after a finite time $t > 0$ (or even instantaneously) $\int_\Omega \rho_1(t, x) \dd x < \int_\Omega \rho_1(0, x) \dd x$. This phenomenon is called \textit{gelation} and corresponds to the appearance of a macroscopic object. In this paper we mainly consider additive bounds on the coagulation rate. That means that $a_{i, j} \lesssim (i + j)^\gamma$ for some $\gamma \in \R$. For the homogeneous case (without diffusion and where $\Omega = \{0\}$), it is well-known that no gelation occurs when $\gamma \leq 1$ but that gelation can occur if $\gamma > 1$ (see for instance the survey \cite{laurencotmischler2004}). One of our goals is to study whether the absence of gelation still holds when $\gamma\leq1$ for the diffusive case.\medskip

Finally, we recall that thanks to the symmetry assumption \ref{ass:symmetry}, we can write (at the formal level) a weak formulation for $Q_i$ and $F_i$ under the form
\begin{equation}\label{eq:weak}
  \sum_{i=1}^{+\infty} \varphi_i\,Q_i(c) = \frac{1}{2}\sum_{i=1}^{+\infty}\sum_{j=1}^{+\infty}a_{i,j}\, c_i\, c_j\,(\varphi_{i+j}-\varphi_i-\varphi_j),
\end{equation}
\begin{equation}\label{eq:weakF}
  \sum_{i=1}^{+\infty} \varphi_i\,F_i(c) = - \sum_{i=2}^{+\infty}B_ic_i\left(\phi_i - \sum_{j = 1}^{i - 1}\beta_{i, j}\phi_j\right), 
\end{equation}
where $\phi : \N^* \to \R$ is any function such that the left-hand size is absolutely convergent, see for instance \cite[Section 2.3.1]{banasiaklamb2019} in the continuous setting.\medskip

In this paper, we consider a smooth framework and classical solutions.
\begin{defi}[Classical solution to \eqref{eq:coagulation}]
  Let $T > 0$, $d\in\N^*$ and let $\Omega$ be a smooth ($\fsC^2$) bounded domain of $\R^d$. We define a classical solution to \eqref{eq:coagulation} as a family of functions $(c_i)_{i \in\N^*}$ such that
  \begin{itemize}
    \item for any $i\in\N^*$, one has $c_i \in \fsC^2([0, T]\times\overline\Omega)$,
    \item for any $i\in\N^*$, $x\in\Omega$, $t\in [0, T]$, the quantities $Q_i(c)(x, t)$ and $F_i(c)(x, t)$ defined by \eqref{eq:Q} and \eqref{eq:F} are absolutely convergent. Furthermore, the functions $(t, x)\mapsto Q_i(c)(t, x)$, $F_i(c)(t, x)$ are continuous on $[0, T]\times\overline\Omega$,
    \item for any $i\in\N^*$, the equation \eqref{eq:coagulation} is satisfied pointwise on $[0, T]\times\Omega$.
  \end{itemize}  
\end{defi}

\subsection{Main results}
In this paper, we are primarily interested in the case when $\Omega$ is a bounded domain of $\R^d$ for $d\leq 2$.
\begin{theo}\label{thm:main2}
  Let $\Omega$ be a bounded smooth domain of $\R^d$, $d\leq 2$. Let $(c_{i, \init})_{i\in\N^*}$ be smooth initial data (that is belonging to $\fsC^\infty(\overline\Omega)$), compatible with the boundary condition, such that for any $k, l \in \N$, there exists some constant $C_{k, l}$ such that for any $i\in\N^*$, $\|i^k c_{i, \init}\|_{\fsC^l(\overline\Omega)} \leq C_{k, l}$. We assume that \ref{ass:symmetry} - \ref{ass:bounda} - \ref{ass:boundB} - \ref{ass:massB} and  \ref{ass:boundd} hold. \medskip

  Then there exists a unique classical solution to \eqref{eq:coagulation} on $\R_+\times\Omega$. Furthermore, for any $T > 0$, $k, l \in\N$, there exists a constant $C'_{k, l}$ depending on $k$, $l$, $\Omega$, $d$, $T$, the parameters of the problem, and the initial data such that for any $i\in\N^*$, $\|i^k c_i\|_{\fsC^l([0, T]\times\overline\Omega)} \leq C'_{k, l}$. Moreover, no gelation occurs, that is the total mass is conserved:
  \begin{equation*}
    \forall t \in \R_+, \, \int_\Omega\sum_{i = 1}^{+\infty} ic_i(t) = \int_\Omega\sum_{i = 1}^{+\infty} ic_{i, \init}.
  \end{equation*}
\end{theo}

Using a different method, we then provide a result in dimension $d\leq 4$, under much more stringent conditions on the coagulation rate and the fragmentation rate. For some $\eta > 0$, we replace \ref{ass:bounda} and \ref{ass:boundB} by

\begin{enumerate}[label={$(\text{A'}\arabic*)_\eta$}]
    \setcounter{enumi}{1}
  \item\label{ass:bounda'} \textit{Strong decay of the coagulation rate:}
  $$\exists C_a > 0, \, \forall i, j \in \N^*, \qquad \, 0 \leq a_{i, j} \leq C_a\frac{ij}{(i + j)^{2+{\eta}}}.$$
  \item\label{ass:boundB'} \textit{Strong decay of the fragmentation rate:}
  $$\exists C_B > 0, \, \forall i \in \N^*, \qquad \, 0 \leq B_{i} \leq C_B\frac{1}{i^{1 + \eta}}.$$
\end{enumerate}

In this case, we call parameters of the problem the constants $\eta$, $C_a$, $C_B$, $\underline{D}$ and $\overline D$.

\begin{theo}\label{thm:maind}
  Let $\Omega$ be a bounded smooth domain of $\R^d$, $d\leq 4$, let $\eta > 0$. Let $(c_{i, \init})_{i\in\N^*}$ be smooth initial data, compatible with the boundary condition, such that for any $k, l \in \N$, there exists a constant $C_{k, l}$ such that for any $i\in\N^*$, $\|i^k c_{i, \init}\|_{\fsC^l(\overline\Omega)} \leq C_{k, l}$. We assume that \ref{ass:symmetry} - \ref{ass:bounda'} - \ref{ass:boundB'} - \ref{ass:massB} and \ref{ass:boundd} hold.\medskip
  
  Then, there exists a unique classical solution to \eqref{eq:coagulation} on $\R_+\times\Omega$. Furthermore, for any $T > $, $k, l \in\N$, there exists a constant $C'_{k, l}$ depending on $k$, $l$, $\Omega$, $d$, $T$, the parameters of the problem and the initial data, such that for any $i\in\N^*$, $\|i^k c_i\|_{\fsC^l([0, T]\times\overline\Omega)} \leq C'_{k, l}$. Moreover, no gelation occurs, that is the total mass is conserved:
  \begin{equation*}
    \forall t \in \R_+, \, \int_\Omega\sum_{i = 1}^{+\infty} ic_i(t) = \int_\Omega\sum_{i = 1}^{+\infty} ic_{i, \init}.
  \end{equation*}
\end{theo}

\begin{rema}
  The assumption that the initial data are smooth and decay faster in $i$ than any polynomials rules out a number of cases. However, a careful reading of the proof shows that one can assume less initial regularity (and in particular only a finite number of moments) at the cost of obtaining less regularity on the solutions.
\end{rema}

Regarding \cref{thm:main2}, our proof follows the duality approach of \cite{canizodesvillettes2010, bredendesvillettes2017}. However, in dimension $2$, we are able to identify another duality problem and by using both space and time regularity, we can stay in a $\fsL^2$ framework for higher moments. More precisely, we prove the
\begin{prop}\label{thm:duality-lemma}
   Let $T > 0$ and let $\Omega$ be a $\fsC^2$ bounded domain of $\R^d$, $d\leq 2$. Let $\mu = \mu(t, x) : [0, T]\times\overline\Omega \to [a, b]$ be a smooth coefficient, where $0 < a < b < +\infty$ are two constants. Let $p > 2$ and let $f$ a smooth function $[0, T]\times\overline\Omega\to \R$ and $u_\init$ a smooth function $\overline\Omega\to \R_+$. Let $u \geq 0$ be a $\fsC^2$ function such that
  \begin{equation}
    \left\{
    \begin{aligned}
      &\partial_t u - \Delta (\mu u) \leq uf \quad
      &&\text{in } [0, T]\times\Omega,\\
      &\nabla (\mu u) \cdot n = 0 \quad
      &&\text{on } [0, T]\times\partial\Omega,\\
      &u(0) = u_\init \quad
      &&\text{in }\Omega.
    \end{aligned}
    \right.
  \end{equation} 
  Then, there exists a constant $C_p$ depending on $\|f\|_{\fsL^p([0, T]\times\Omega)}$, $T$, $\Omega$, $d$, $p$, $a$ and $b$ such that
  \begin{equation*}
    \|u\|_{\fsL^2([0, T]\times\Omega)} \leq C_p\|u_\init\|_{\fsL^2(\Omega)}.
  \end{equation*}
\end{prop}

We emphasize that in this result the constant $C_p$ depends on $\mu$ only through the upper and lower-bounds $a$ and $b$.\medskip

The proof of \cref{thm:maind} relies on methods used in \cite{boutondesvillettes2025} to study reaction-diffusion systems. Instead of working in $\fsL^2([0, T]\times\Omega)$, we show that the local mass $\rho_1 := \sum_{i = 1}^{+\infty} ic_i$ enjoys some negative Hölder regularity. This initial estimate turns out to be sufficient to show smoothness of all the moments. We note that this method bears some resemblance to the works of Hammond and Rezakhanlou \cite{hammond2007, rezakhanlou2010, rezakhanlou2014}, where the solution to the Poisson equation with right-hand side $\rho_1$ is shown to be bounded in $\fsL^\infty([0, T]\times\Omega)$.

\subsection{Comparison with the literature}

Previous works on coagulation-diffusion systems have studied the questions of regularity and absence of gelation. In a series of works of Rezakhanlou and coauthors \cite{hammond2007, rezakhanlou2010, rezakhanlou2014} and of Cañizo, Desvillettes, Fellner and coauthors \cite{canizodesvillettes2010, canizodesvillettes2014, bredendesvillettes2017} both absence of gelation and regularity of the solutions to the system \eqref{eq:coagulation} were proved in all dimensions.\medskip

However, additional hypotheses had to be made in comparison to \ref{ass:bounda} and \ref{ass:boundd}. Indeed, it was necessary to assume at least that for any fixed $j\in\N^*$, one has $\frac{a_{i, j}}{i} \to 0$ as $i \to \infty$. This assumption did not allow to treat the critical case $a_{i, j} \approx  i + j$, which is known to preserve mass in the homogeneous setting (see for instance the survey \cite{laurencotmischler2004}). We are thus able to close the gap in dimension $d\leq 2$ between the homogeneous and the spatially inhomogeneous diffusive setting. We point out that this critical rate is also of importance in applications for colliding water droplets \cite{golovin1963}, for particles in laminar shear field \cite{smoluchowski1918} (see also \cite{drake1972}), or in polymerization processes \cite{cohenbenedek1982, dacosta2015}. We also notice that a theory of weak or mild solutions for the diffusive case has been proposed in \cite{wrzosek1997, laurencot2002}, however the critical case is not included in those results, since similar decay assumptions to those in the previously cited works had to be imposed. Thus, even the existence of global weak solutions in dimension $d\leq 2$ was an open problem to the knowledge of the author (see for instance the discussion raised in \cite[p.407]{wrzosek2004}). \medskip

We point out that in \cref{thm:main2} the assumptions on the fragmentation rates are very weak. The assumption \ref{ass:massB} ensures that the fragmentation process neither creates nor destroys mass, and we only assume polynomial growth of the fragmentation rate \ref{ass:boundB}. We emphasize that such a weak assumption is possible because we work in a discrete framework. In the continuous setting, where the mass is allowed to vary continuously in $(0, +\infty)$, one may observe the dual phenomenon of gelation: shattering (a loss of mass due to the appearance of excessively small particles). Furthermore, we note that in our setting, one can take $B \equiv 0$. Hence, contrary to some works such as \cite{shindin2026}, the fragmentation process does not balance coagulation in order to avoid gelation.\medskip

Another assumption that had to be made in previous works in the inhomogeneous setting in order to get classical solutions was that the diffusion coefficients $d_i$ had to be non-increasing with $i$, or, at least, to satisfy a weaker regularity condition (see \cite{rezakhanlou2014} for more details). We are able to remove this hypothesis in both theorems. However in dimension $d = 3, 4$, we have to impose a very fast decay rate for the coefficients \ref{ass:bounda'} and \ref{ass:boundB'}. To the best of the author's knowledge, existence of global classical solutions was not proved previously, even for strongly decaying kernels, without additional hypotheses on the diffusion coefficients $d_i$ (see for instance the case of bounded coagulation kernels in \cite{amann2000, wrzosek2002}). We nevertheless note that \cite{canizodesvillettes2010} proves absence of gelation with the mere assumption \ref{ass:boundd}.

In the continuous and homogeneous setting, we note that coagulation equations satisfying \ref{ass:bounda'} have been studied \cite{clarkkatsouros1999, laurencot2018}, however the blow-up of the coagulation coefficient for very small particles in the continuous setting of those works raises other challenges. \medskip

Finally, we point out that in a number of applications it is reasonable to expect that $d_i \to 0$ as $i\to +\infty$. Indeed, as $i$ grows, the mass of a cluster of size $i$ is expected to grow, and hence one can suppose that it diffuses more slowly. Our method is unable to cover this case contrary to the previous works of Rezakhanlou and coauthors \cite{hammond2007, rezakhanlou2010, rezakhanlou2014} or \cite{canizodesvillettes2010a}, where absence of gelation is proved under an additional decay hypothesis on $a_{i, j}$. We also point to a recent renewed interest in this direction \cite{dasjaiswal2026, shindin2026}.

\subsection{Plan of the paper and notations}
In \cref{sec:proof2}, we prove \cref{thm:duality-lemma} and we then apply it to complete the proof of \cref{thm:main2}. We then prove \cref{thm:maind} in \cref{sec:proofd}. Finally, in Appendix~\ref{sec:smoothing}, we give a short proof of a well-known smoothing estimate for the heat equation, for which we could not find a precise proof in the literature.\medskip

Recall that we defined $\N^* := \N\setminus\{0\}$. For any $k\in\N$, we define the moments $\rho_k := \sum_{i = 1}^{+\infty}i^k c_i$, we also naturally define $\rho_{k, \init}$ for the initial value of the moments. For some $k\in\N$, we denote the sequence $C_k := (i^k c_i)_{i\in\N^*}$, we extend it by $0$ for $i\leq 0$ if needed. Finally, for a bounded domain $\Omega$ and some $T > 0$, we introduce the notation $\Omega_T := [0, T]\times \Omega$.

\section{Proof of \texorpdfstring{\cref{thm:main2}}{Theorem 1}}\label{sec:proof2}
\subsection{A duality lemma}

We begin by analyzing an abstract parabolic problem in non-divergence form. We provide estimates that do not depend on the regularity of the coefficient of the equation.

\begin{prop}\label{thm:dual-estimate}
  Let $T > 0$ and let $\Omega$ be a $\fsC^2$ bounded domain of $\R^d$, $d\leq 2$. Let $\mu = \mu(t, x) : [0, T]\times\overline\Omega \to [a, b]$ be a smooth coefficient, where $0 < a < b < +\infty$ are two constants. Let $p > 2$ and $f, H$ be two smooth functions $[0, T]\times\overline\Omega \to \R$. Let $w$ be a $\fsC^2([0, T]\times\overline\Omega)$ solution to
  \begin{equation}\label{eq:dual}
    \left\{
    \begin{aligned}
      &\partial_t w - \mu \Delta w = wf + H\quad
      &&\text{in } \Omega_T,\\
      &\nabla w \cdot n = 0 \quad
      &&\text{on } [0, T]\times\partial\Omega.\\
      &w(0) = 0 \quad
      &&\text{in }\Omega.
    \end{aligned}
    \right.
  \end{equation} 
  
  Then, there exist constants $C_T, \, C'_T$ that depend on $a$, $b$, $\Omega$, $d$, $T$, $p$ such that for $q := \frac{2p}{p - 2}$, one has  
  \begin{equation*}
    \|w\|_{\fsL^\infty([0, T], \fsL^q(\Omega))}\leq C'_T e^{C_T\|f \|_{\fsL^p(\Omega_T)}^q}\|H\|_{\fsL^2(\Omega_T)}.
  \end{equation*}
\end{prop}

\begin{proof}
  Let $0 < t < T$, since we are in a smooth framework, one can test \eqref{eq:dual} with $\frac{\partial_s w}{\mu}$, to obtain
  \begin{align*}
    \int_0^t\int_\Omega \frac{|\partial_s w|^2}{\mu} + \frac12 \int_\Omega |\nabla w(t)|^2
    &= \int_0^t\int_\Omega \left(wf\frac{\partial_s w}{\mu} + H \frac{\partial_s w}{\mu}\right)\\
    &\leq \frac12 \int_0^t\int_\Omega \left(\frac{w^2f^2 + H^2}{\mu}\right) + \frac12\int_0^t\int_\Omega\frac{(\partial_s w)^2}{\mu}.
  \end{align*}
  To obtain the last line, we used Young's inequality.

  Using the bounds on $\mu$, we obtain
  \begin{equation}\label{eq:energy}
    \frac{1}{b} \int_0^t\int_\Omega|\partial_s w|^2 + \int_\Omega |\nabla w(t)|^2
    \leq \frac{1}{a}\int_0^t\int_\Omega \left(w^2f^2 + H^2\right).
  \end{equation}
  
  Then, using Jensen inequality 
  \begin{align}\label{eq:meancontrol}
    \int_\Omega |w(t, x)|^2
    &= \int_\Omega \left|\int_0^t \partial_s w\right|^2\nonumber\\
    &\leq T \int_0^t\int_\Omega |\partial_s w|^2\\
    &\leq \frac{bT}{a}\int_0^t\int_\Omega \left(w^2f^2 + H^2\right).\nonumber
  \end{align}

  We notice by Hölder's inequality that
  \begin{equation}\label{eq:holderwf}
    \|wf\|_{\fsL^2(\Omega_t)} \leq \|w\|_{\fsL^q(\Omega_t)} \|f\|_{\fsL^p(\Omega_T)},
  \end{equation}
  where we used that $\frac{1}{p} + \frac1q = \frac12$.

  One can then use a Sobolev inequality in dimension $d\leq 2$ to deduce from \eqref{eq:energy}, \eqref{eq:meancontrol} and \eqref{eq:holderwf} that
  \begin{align}\label{eq:Poincare-Sobolev}
    \|w(t)\|_{\fsL^q(\Omega)}
    &\leq C_\Omega\left(\|\nabla w(t)\|_{\fsL^2(\Omega)} + \|w(t)\|_{\fsL^2(\Omega)}\right)\nonumber\\
    &\leq C_\Omega\left(\sqrt{\frac1a} + \sqrt{\frac{bT}{a}}\right) \left[\|w\|_{\fsL^{q}(\Omega_t)} \|f \|_{\fsL^p(\Omega_T)} + \|H\|_{\fsL^2(\Omega_T)}\right],
  \end{align}
  where $C_\Omega$ is a constant that depends on $\Omega$, $d$ and $q$.

  Thus, defining $\phi = \int_\Omega w^q$ and using \eqref{eq:Poincare-Sobolev}, we obtain for a constant $K_1 := C_\Omega\left(\sqrt{\frac1a} + \sqrt{\frac{bT}{a}}\right)$
  \begin{align*}
    \phi(t)^{1/q}
    &\leq K_1\left(\|w\|_{\fsL^{q}(\Omega_t)} \|f \|_{\fsL^p(\Omega_T)} + \|H\|_{\fsL^2(\Omega_T)}\right)\\
    &\leq K_1 \left[\left(\int_0^t \phi\right)^{1/q} \|f \|_{\fsL^p(\Omega_T)} + \|H\|_{\fsL^2(\Omega_T)}\right].
  \end{align*}

  Raising to the power $q$, we deduce that
  \begin{align*}
    \phi(t)
    &\leq \left(2K_1\right)^q\left(\|f \|_{\fsL^p(\Omega_T)}^q\int_0^t \phi + \|H\|_{\fsL^2(\Omega_T)}^q\right).
  \end{align*}
  
  Gronwall's inequality then gives for $0 \leq t \leq T$ the bound
  \begin{equation*}
    \phi(t) \leq \left(2K_1\right)^qe^{TK_1^q 2^q\|f \|_{\fsL^p(\Omega_T)}^q}\|H\|_{\fsL^2(\Omega_T)}^q.
  \end{equation*}

  Thus, we have proved that
  \begin{equation*}
    \|w\|_{\fsL^\infty([0, T], \fsL^q(\Omega))}\leq 2K_1e^{\frac{TK_1^q 2^q}{q}\|f \|_{\fsL^p(\Omega_T)}^q}\|H\|_{\fsL^2(\Omega_T)}.
  \end{equation*}
\end{proof}

We can now return to the problem of interest.

\begin{proof}[Proof of \cref{thm:duality-lemma}]
  Let $H := H(t, x)\geq 0$ be a smooth function. We consider the non-negative solution to the dual backward in time problem 
  \begin{equation*}
    \left\{
    \begin{aligned}
      &\partial_t w + \mu \Delta w = - wf - H\quad
      &&\text{in } \Omega_T,\\
      &\nabla w \cdot n = 0 \quad
      &&\text{on } [0, T]\times\partial\Omega,\\
      &w(T) = 0 \quad
      &&\text{in }\Omega.
    \end{aligned}
    \right.
  \end{equation*}
  Since we are working in a $\fsC^2$ setting, the existence of $w$ is standard (see for instance the discussion in the Appendix of \cite{bredendesvillettes2017} for a similar problem, if one wants to work with rough coefficients).\medskip

  After reversing time, we can apply the estimates developed in \cref{thm:dual-estimate}, and deduce that for some constant $C$ depending on $a$, $b$, $T$, $d$, $\Omega$, $p$, and $\|f\|_{\fsL^p(\Omega_T)}$, one has
  $$\|w(0)\|_{\fsL^2(\Omega)} \leq |\Omega|^{\frac{1}{p}} \|w(0)\|_{\fsL^q(\Omega)} \leq C \|H\|_{\fsL^2(\Omega)},$$
  where $q > 2$ is defined in \cref{thm:dual-estimate}.

  We then note that (using $w\geq 0$)
  \begin{align*}
    \frac{\dd }{\dd t} \int_\Omega wu 
    &= \int_\Omega (\partial_t w)u + w (\partial_t u)\\
    &\leq \int_\Omega (- \mu \Delta w - fw - H)u + w(\Delta (\mu u) + fu)\\
    &= -\int_\Omega Hu.
  \end{align*}

  Hence, using $w(T) = 0$
  \begin{equation*}
    \int_0^T\int_\Omega H u \leq \int_\Omega w(0)u_\init \leq \|w(0)\|_{\fsL^2(\Omega)} \|u_\init\|_{\fsL^2(\Omega)} \leq C \|H\|_{\fsL^2(\Omega)}\|u_\init\|_{\fsL^2(\Omega)}. 
  \end{equation*}

  Using $u\geq 0$ and the density of smooth functions in $\fsL^2(\Omega_T)$, a duality argument yields that
  \begin{equation*}
    \|u\|_{\fsL^2(\Omega_T)} \leq C \|u_\init\|_{\fsL^2(\Omega)}.
  \end{equation*}
\end{proof}

\begin{rema}
  We stated \cref{thm:duality-lemma} in a smooth ($\fsC^2$) setting, since we only use it on an approximate problem. However, it is clear that it holds under weaker hypotheses, in particular $\mu$ may be assumed to be only measurable.
\end{rema}

\subsection{End of the proof of \texorpdfstring{\cref{thm:main2}}{Theorem 1}}\label{sec:endproof}
We can now apply this duality lemma to the study of system \eqref{eq:coagulation}.\medskip

\begin{proof}[Proof of \cref{thm:main2}]
  We will prove the existence of a solution on $[0, T]\times\Omega$ for any $T > 0$. It can be easily checked that these solutions are identical on the shared domain of definition and define thus a global solution. Therefore, we consider in the following an arbitrary $T > 0$.

  To make all computations rigorous, we consider the following truncated system: for some $N > 0$, we define $a^N_{ij} := a_{i, j}\indic{i + j \leq N}$, we also restrict the system to $c_i$ with $i\leq N$. Classical parabolic theory then ensures the existence for any $T > 0$ of a \textit{non-negative smooth solution to the truncated system}. The estimates below are uniform in $N$, and therefore provide uniform smoothness estimates for the relevant quantities (we note that all the functions that we consider being smooth, we can apply \cref{thm:duality-lemma}). An application of the Arzelà--Ascoli theorem allows us to deduce the existence of a smooth solution to \eqref{eq:coagulation}. Uniqueness then follows immediately. Regarding the absence of gelation, it comes from the uniform control of a higher moment (as $\rho_2$) in any Lebesgue space. Since this approximation procedure is standard (see for instance \cite{canizodesvillettes2010}), we do not describe it explicitly and we focus on uniform smoothness of the solution.

  Using the weak formulation \eqref{eq:weak} and \eqref{eq:weakF} (as well as \ref{ass:massB}), we note that $\sum_{i = 1}^{+\infty} iQ_i=\sum_{i = 1}^{+\infty} iF_i = 0$. Summing \eqref{eq:coagulation} (with a factor $i$) we see that the following equation is satisfied
  \begin{equation*}
    \partial_t \left(\sum_{i=1}^{+\infty} i c_i\right) - \Delta \left(\sum_{i=1}^{+\infty} id_i c_i\right) = 0.
  \end{equation*}
  Hence, if we define $\mu_1 := \frac{\sum_{i=1}^{+\infty} id_i c_i}{\sum_{i=1}^{+\infty} i c_i}$, we obtain the equation for $\rho_1$
  \begin{equation*}
    \left\{
      \begin{aligned}
        &\partial_t \rho_1 - \Delta(\mu_1\rho_1) = 0\quad
        &&\text{in } \Omega_T,\\
        &\nabla(\mu_1 \rho_1)\cdot n = 0\quad
        &&\text{on } [0, T]\times\partial\Omega,\\
        &\rho_1(0) = \rho_{1,\init}\quad
        &&\text{in }\Omega.
      \end{aligned}
    \right.
  \end{equation*}
  Furthermore, using the non-negativity of the $c_j$, one can check that $\underline D \leq \mu_1 \leq \overline D$ thanks to \ref{ass:boundd}. By the improved duality lemma of \cite[Proposition 1.1]{canizodesvillettes2014}, we deduce that there exist some $p>2$ and $C > 0$ depending on $\Omega$, $d$, $\underline D$, $\overline D$, $T$ and $\|\rho_{1, \init}\|_{\fsL^p(\Omega)}$ such that $\|\rho_1\|_{\fsL^p(\Omega_T)} \leq C$. \medskip

    Let $k\in\N\setminus\{0, 1\}$, then by the weak formulation \eqref{eq:weak} and \eqref{eq:weakF}
  \begin{align*}
    \partial_t \left(\sum_{i=1}^{+\infty} i^k c_i\right)- \Delta\left(\sum_{i=1}^{+\infty} d_i i^k c_i\right)
    &= \frac12 \sum_{i = 1}^{+\infty}\sum_{j=1}^{+\infty} a_{i, j}\left((i + j)^k - i^k - j^k\right)c_i c_j\\
    &\phantom{=} + \sum_{i = 2}^{+\infty}B_{i}c_i \left(\sum_{j = 1}^{i - 1}\beta_{i, j}j^k - i^k\right)\\
    & =: K + F.
  \end{align*}

  Concerning the fragmentation part, we have
  \begin{align*}
    F
    &\leq \sum_{i = 2}^{+\infty}B_{i}c_i \left(\sum_{j = 1}^{i - 1}\beta_{i, j}j - i\right) i^{k - 1} = 0,
  \end{align*}
  where we used \ref{ass:massB} to obtain the last equality.\medskip

  Using \ref{ass:bounda}, we can upper bound the coagulation part
  \begin{align*}
    K
    &\leq \frac{C_a}{2} \sum_{i = 1}^{+\infty}\sum_{j=1}^{+\infty} (i + j)\left((i + j)^k - i^k - j^k\right)c_i c_j.
  \end{align*} 

  We use Newton's binomial formula and Young's inequality under the form $(i + j)i^lj^{k - l} \leq C_k(i^kj + ij^k)$ for any $i, \, j\in\N^*$, $1\leq l \leq k - 1$ and for some constant $C_k > 0$, to deduce that
  \begin{align*}
    \frac{C_a}{2} \sum_{i = 1}^{+\infty}\sum_{j=1}^{+\infty} (i + j)\left((i + j)^k - i^k - j^k\right)c_i c_j
    &\leq C \sum_{i = 1}^{+\infty}\sum_{j=1}^{+\infty} (i^kj + ij^k)c_i c_j\\
    &\leq C \rho_k \rho_1,
  \end{align*} 
  where $C$ is a constant that may change from line to line and that depends on $k$ and $C_a$.\medskip

  We obtain the following equation for $\rho_k$
    \begin{equation*}
    \left\{
      \begin{aligned}
        &\partial_t \rho_k - \Delta(\mu_k\rho_k) \leq C\rho_1\rho_k\quad
        &&\text{in } \Omega_T,\\
        &\nabla (\mu_k \rho_k)\cdot n = 0\quad
        &&\text{on } [0, T]\times\partial\Omega,\\
        &\rho_k(0) = \rho_{k,\init}\quad
        &&\text{in }\Omega,
      \end{aligned}
    \right.
  \end{equation*}
  where $\mu_k := \frac{\sum_{i=1}^{+\infty} i^kd_i c_i}{\sum_{i=1}^{+\infty} i^k c_i}$ satisfies the bound $\underline D\leq \mu_k \leq \overline D$.\medskip

  We apply \cref{thm:duality-lemma} to deduce that for some constant $C_k$ that depends on $\Omega$, $d$, $\underline D$, $\overline D$, $C_a$, $k$, $\|\rho_{k, \init}\|_{\fsL^2(\Omega)}$ and $\|\rho_{1, \init}\|_{\fsL^p(\Omega)}$ (where $p$ is defined previously), one has $\|\rho_k\|_{\fsL^2(\Omega_T)} \leq C_k$.\medskip

  We can now improve these estimates in order to obtain the full regularity stated in the theorem. We first show that all the moments are bounded in $\fsL^\infty(\Omega_T)$.
  \medskip

  We define the sequence $q_r$ by $q_1 := p$ and for $r\in\N^*$, $\frac{1}{q_{r + 1}} := \frac{2}{q_r} - \frac{1}{2}$. Since $q_1 > 2$, we see that this sequence becomes non-positive in a finite number of steps. \medskip

  We first notice that \ref{ass:massB} implies that $\beta_{i + j, i} \leq i + j$ for any $i,\, j\in\N^*$. Let $i \in\N^*$ and $l \in \N$. Using \ref{ass:bounda} and \ref{ass:boundB}, we deduce
  \begin{align*}
    (\partial_t - d_i \Delta)i^lc_i
    &\leq i^l\left(Q_i^+ + F_i^+\right)\\
    &\leq \left(\frac{C_a}{2} + C_B\right) \left(i^{l + 1} \sum_{j = 1}^{i - 1} c_j c_{i - j} + \sum_{j = 1}^{+\infty} (i + j)^{l + m + 1}c_{i + j}\right).
  \end{align*}

  We now use the following decomposition
  \begin{equation*}
    i^{l + 1} \sum_{j = 1}^{i - 1} c_j c_{i - j} = \sum_{j = 1}^{\lfloor \frac{i}{2}\rfloor} c_j (i^{l + 1}c_{i - j}) + \sum_{j = \lfloor \frac{i}{2}\rfloor + 1}^{i - 1} (i^{l+1}c_j) c_{i - j}.
  \end{equation*}

  We note that $i \leq 2 (i - j)$ if $j \leq \lfloor \frac{i}{2}\rfloor$ and $i \leq 2 j$ if $j \geq \lfloor \frac{i}{2}\rfloor + 1$. Hence, we obtain
  \begin{align*}
    i^{l + 1} \sum_{j = 1}^{i - 1} c_j c_{i - j}
    &\leq 2^{l +1}\sum_{j = 1}^{\lfloor \frac{i}{2}\rfloor} c_j ((i - j)^{l + 1}c_{i - j})
    + 2^{l +1}\sum_{j = \lfloor \frac{i}{2}\rfloor + 1}^{i - 1} (j^{l+1}c_j) c_{i - j}\\
    &\leq 2^{l +1}\left[\max_{0 < j \leq \lfloor \frac{i}{2}\rfloor} c_{j} \right]\sum_{j = 1}^{\lfloor \frac{i}{2}\rfloor} c_{i - j} (i - j)^{l + 1}
    + 2^{l +1}\left[\max_{\lfloor \frac{i}{2}\rfloor + 1\leq j \leq i - 1} c_{i - j} \right]\sum_{j = \lfloor \frac{i}{2}\rfloor + 1}^{i - 1} j^{l + 1}c_j\\
    &\leq 2^{l +2}\rho_1\rho_{l + 1}.
  \end{align*}

  To obtain the last line, we used that $c_j \leq \rho_0\leq \rho_1$ for any $j\in\N^*$. We then notice that $\sum_{j = 1}^{+\infty} (i + j)^{l + m + 1}c_{i + j}\leq \rho_{l + m + 1}$. Thus, defining $K := 2^{l + 2}\left(\frac{C_a}{2} + C_B\right)$, we have shown that
  \begin{align}\label{eq:boundci}
    (\partial_t - d_i \Delta)i^lc_i
    \leq K \left(\rho_{l + 1}\rho_1 + \rho_{l + m + 1}\right).
  \end{align}

  Since $\rho_{l + m + 1}, \, \rho_{l + 1} \in\fsL^{2}(\Omega_T)$ and $\rho_1\in\fsL^p(\Omega_T)$, we deduce that $0 \leq \rho_{l + 1}\rho_1 + \rho_{l + m + 1}$ is bounded in $\fsL^{p_1}(\Omega_T)$, where $\frac{1}{p_1} = \frac{1}{p} + \frac12$. We recall that in dimension $d\leq 2$, the solution to the Neumann heat equation with a forcing in $\fsL^\kappa{(\Omega_T)}$ is bounded in $\fsL^\eta{(\Omega_T)}$ with $\frac{1}{\eta} = \frac{1}{\kappa} - \frac12$ (provided that $1 < \kappa < 2$), see \cref{thm:smoothing} in Appendix~\ref{sec:smoothing}. Using \ref{ass:boundd} and the non-negativity of $c_i$, one deduces that there exists a constant $C_l > 0$ such that for all $i\in\N^*$, one has $\|i^l c_i\|_{\fsL^{q_1}(\Omega_T)} \leq C_l$. In particular $\rho_{l - 2}$ is bounded in $\fsL^{q_1}(\Omega_T)$.
  
  We can now iterate this argument. We assume that all the $\rho_k$ are bounded in $\fsL^{q_r}(\Omega_T)$. Thanks to \eqref{eq:boundci}, for any $l\in\N$, $i^lc_i \geq 0$ solves a Neumann heat equation with a right-hand side, whose positive part is bounded by $\rho_{l + 1}\rho_1 + \rho_{l + m + 1}$. This forcing lies in $\fsL^{\frac{q_r}{2}}(\Omega_T)$ by the induction hypothesis. It follows that $i^lc_i$ is uniformly bounded in $i$ in $\fsL^{q_{r+1}}(\Omega_T)$ (the regularity argument is valid as long as $q_{r} < 4$). Iterating eventually one last time (or twice if we reach $q_r = 4$), we deduce that all the $\rho_k$ are bounded in $\fsL^\infty(\Omega_T)$, by the smoothing effect of the heat equation.\medskip

  We can now iteratively show that for any $r\in \N$, $k \geq 1, 1 < p < +\infty$, there exists a constant $C_{k, p, r}$ such that for all $i \in \N^*$, $\|i^kc_i\|_{W^{r, p}(\Omega_T)} \leq C_{k, p, r}$, where $W^{r, p}(\Omega_T)$ is a parabolic Sobolev space. We also refer to \cite{bredendesvillettes2017} for a similar induction. We consider again the equation satisfied by $i^k c_i$
  $$\left(\partial_t  - d_i \Delta \right)i^k c_i = i^k\left(Q_i + F_i\right). $$
  By the moment bounds in $\fsL^\infty(\Omega_T)$, $i^k Q_i$ and $i^k F_i$ are bounded uniformly in $i$ in all $\fsL^p(\Omega_T)$ for $1 < p < +\infty$. Indeed, we have already shown that the positive part $i^kF_i^+$ and $i^kQ_i^+$ are bounded by well-chosen moments. The bound on $|i^kF_i^-| = B_ii^kc_i \leq \rho_{k + m}$ is immediate. Finally, we can bound $Q_i^-$
  \begin{align*}
    |i^kQ_i^-|\leq i^kc_i\sum_{j = 1}^{+\infty}(i + j) c_j= i^{k + 1}c_i\rho_0 + i^kc_i\rho_1 \leq 2\rho_1\rho_{k +1}.
  \end{align*}
  
  Using classical regularity theory, one deduces that $i^k c_i$ is bounded (uniformly in $i$) in $W^{2, p}(\Omega_T)$ for any $1 < p < +\infty$.
  
  We deduce from this estimate that $i^k Q_i$ and $i^k F_i$ are bounded uniformly in $i$ in $\fsW^{2, p}(\Omega_T)$ for any $1 < p < +\infty$. We detail the computations on the example of the term $i^k Q_i^-$. For any \mbox{$a, \, b \in  \{x_1, \, x_2\}$} (for $d=2$), we can compute $$\partial_a\partial_b (i^k Q_i^-) = -i^k\sum_{j = 1}^{+\infty} a_{i, j} \partial_a\partial_b (c_i c_j).$$
  We notice that $\partial_a\partial_b (c_i c_j) = (\partial_a\partial_b c_i) c_j + (\partial_a\partial_b c_j) c_i + \partial_a c_i \partial_b c_j + \partial_b c_i \partial_a c_j$. By the previous estimate, for any $1 < p < +\infty$,  $i^{k+1}\partial_a\partial_b c_i$ is bounded uniformly in $i$ in $\fsL^p(\Omega_T)$, whereas $j^3 c_j$ is bounded uniformly in $j$ in $\fsL^{\infty}(\Omega_T)$, by the $\fsL^\infty(\Omega_T)$ bound proved previously. Thus, $i^{k+1}(\partial_a\partial_b c_i)j^3c_j$ is bounded uniformly in $i$ and $j$ in $\fsL^p(\Omega_T)$. This bound holds also for the three other terms. We hence deduce that $\partial_a\partial_b (i^k Q_i^-)$ is bounded uniformly in $i$ in $\fsL^p(\Omega_T)$. The same holds for the terms $Q_i^+$, $F_i$ and for the regularity in time.

  Hence, $i^k Q_i$ and $i^k F_i$ are bounded uniformly in $i$ in $\fsW^{2, p}(\Omega_T)$ for any \mbox{$1 < p < +\infty$}. Using once again classical regularity theory, we deduce that the result holds for $r=4$. Iterating this argument, we obtain the result for all $r$ and $1 < p < +\infty$.\medskip

  Finally, using Sobolev embedding, we obtain the desired conclusion.
\end{proof}

\section{Proof of \texorpdfstring{\cref{thm:maind}}{Theorem 2}}\label{sec:proofd}
As previously, we consider an arbitrary $T > 0$. Similarly to the proof of \cref{thm:main2}, we first prove that all the moments are bounded in $\fsL^p(\Omega_T)$ for some $p$ and we then upgrade this estimate to higher regularity. However, since we work in dimension $d \leq 4$, the bootstrap argument works only once we have a bound on the moments in $\fsL^p(\Omega_T)$ for $p > 3$. For this reason, the duality methods used in the previous section are insufficient. We replace them with the two following lemmas.\medskip

The first one is a lemma in the spirit of the Krylov-Safonov theory. It shows that the solution to a parabolic equation in non-divergence form with rough coefficients enjoys some Hölder regularity. Under the assumption that the solution is (pointwise) non-decreasing, an elementary proof is provided in \cite{boutondesvillettes2025}. Furthermore, this assumption allows to treat the optimal range of forcing $f$, whereas the classical Krylov-Safonov theory requires additional integrability. This improvement will be crucial in the following.

\begin{prop}[Theorem 1 from \cite{boutondesvillettes2025}]\label{thm:main-neumann}
  We consider a bounded, $\fsC^2$ domain \(\Omega \subset \R^d\).  Set \(T>0\), and   \(p, q \in [1, \infty]\) such that $\gamma := 2 - \frac2p - \frac{d}q > 0$. We also consider two constants $a_0 > 0, c_0 > 1$. Let \(f \in \fsL^{p}((0,T], \fsL^q(\Omega))\) and let \(a := a(t,x)\) be a measurable coefficient satisfying the bound \(a_0\leq a\leq a_0c_0\). Finally, we consider Lipschitz initial data $w_\init\in\Lip(\Omega)$.

  Let \(w \ge 0\) be a solution of
  \begin{equation*}
    \left\{
      \begin{aligned}
        &a(t, x)\partial_t w(t, x) - \Delta w(t, x) = f(t, x)\quad\text{and }\quad \partial_t w(t, x) \geq 0\quad
        &&\text{in }[0, T]\times\Omega,\\
        &\nabla w \cdot n = 0\quad
        &&\text{on }[0, T]\times\partial\Omega,\\
        &w(0) = w_\init\quad
        &&\text{in } \Omega.
      \end{aligned}
    \right.
  \end{equation*}
  
  Then there exists a constant \(\alpha >0\) depending only on \(\gamma, \, d,\, c_0\), and a constant
  \(C_*\) depending on \(p,\, q,\, d,\,  \Omega,\, T,\, a_0,\, c_0\) such that $w$ lies in \(\fsC^{\alpha}([0,T] \times {\overline{\Omega}} ) \), and, moreover,
  the estimate holds:
  \begin{align*}
    \| w \|_{\fsC^\alpha([0,T] \times {\overline{\Omega}} ) }
    \leq
    C_*    \left(
      \| f \|_{\fsL^p((0,T], \fsL^{q}(\Omega))}
      +
      \| w_\init \|_{\mathrm{Lip}(\Omega)}
    \right).
  \end{align*}
\end{prop}

The second tool that will be needed is an interpolation result in the spirit of the Gagliardo-Nirenberg-Sobolev inequality. However, it replaces the negative Hölder bound by a slightly weaker bound that takes the form of an inequality. We state it in the form needed for our application and we refer to \cite{boutondesvillettes2025} for a more general statement.

\begin{lemm}[Proposition 9 of \cite{boutondesvillettes2025} with $p= 2, \, q= \frac{3 - \alpha}{2 - \alpha}2$]\label{thm:interpolation}
Let $\Omega \subset \R^d$ be a bounded \(\fsC^2\) domain with $d\leq 5$, let $\alpha \in (0,\frac12)$ and let $\eps_1 := \frac{3 - \alpha}{2 - \alpha}2 - 3 > 0$. Then for any $u,w : \Omega \to \R$
  such that $0 \le u \le\Delta w$, it holds that
  \begin{equation}\label{inppalk}
    \|u \|_{\fsL^{3+\eps_1}(\Omega)}
    \le  C_{d, \alpha}\, \| {w} \|_{\fsC^{\alpha}(\overline{\Omega}) }^{\frac{1}{3- \alpha}}\,
    \| \nabla u\|_{\fsL^2(\Omega)}^{\frac{2}{3+\eps_1}}
    +  C_{d,\alpha}\, \| {w} \|_{\fsC^{\alpha}(\overline{\Omega})},
  \end{equation}
  where $C_{d,\alpha}>0$ depends only on $d$, $\alpha$.\medskip
\end{lemm}

Using these lemmas, we can conclude the proof of \cref{thm:maind}. As in the proof of \cref{thm:main2}, we begin with a regularization procedure which allows us to work in a smooth framework with a finite number of equations, before passing to the limit. Existence, uniqueness, and absence of gelation follow from the estimates that we provide. We do not describe this procedure explicitly.

We begin with the
\begin{prop}\label{thm:coagulation-basecase}
  We consider the framework of \cref{thm:maind} and some $T > 0$. Then, there exist $\eps,\, \eps' > 0$ and a constant $C$ depending on the parameters of the problem, $\Omega$, $d$, $T$ as well as on $\|\rho_{\frac{3+{\eta}}{2}, \init}\|_{\fsL^{3}(\Omega)}$ and $\|\rho_{1, \init}\|_{\fsL^{\infty}(\Omega)}$ such that
  $$\|\rho_1\|_{\fsL^{3+\eps}(\Omega_T)}+ \|\rho_1\|_{\fsL^2([0, T], \fsL^{2^*+\eps'}(\Omega))} \leq C,$$
  where \mbox{$2^* := \frac{2d}{d - 2}$} (and any $q < +\infty$ if $d\leq 2$).
\end{prop}

\begin{proof} 
  We first establish an energy estimate for solutions to \eqref{eq:coagulation}. Let $p \in [2, \min(3, 2 + \frac{\eta}{2})[$, we multiply \eqref{eq:coagulation} by $i^{3 + \eta}c_i^{p-1}$, we integrate over $\Omega_t$ and sum over $i$ to obtain
\begin{align*}
    &\frac1p\int_\Omega \sum_{i=1}^{+\infty} (i^{\frac{3 + \eta}{p}} c_i(t))^p - \frac1p\int_\Omega \sum_{i=1}^{+\infty} (i^{\frac{3 + \eta}{p}} c_i(0))^p
    + \underline{D}\frac{4(p - 1)}{p^2}\int_0^t\int_\Omega \sum_{i=1}^{+\infty} |\nabla i^{\frac{3 + \eta}{2}} c_i^\frac{p}{2}|^2\\
    &\leq \int_0^t\int_\Omega \sum_{i=1}^{+\infty}(i^{3 + \eta} c_i^{p- 1}) (Q_i^+ + F_i^+)\\
    &=\frac12 \int_0^t\int_\Omega \sum_{i=1}^{+\infty} (i c_i)^{p - 1}\sum_{j=1}^{i - 1}i^{4 + \eta - p} a_{i-j, j}c_{i-j}c_j + \int_0^t\int_\Omega \sum_{i=1}^{+\infty} (i c_i)^{p - 1}\sum_{j= 1}^{+\infty}i^{4 + \eta - p} B_{i+j}\beta_{i + j, i}c_{i+j}\\
    &=: K + F.
\end{align*}

We first consider the fragmentation part $F$. Using \ref{ass:massB}, we deduce that for any $i, \, j\in\N^*$, one has $\beta_{i + j, i} \leq\frac{i + j}{i}$. Using \ref{ass:boundB'}, we deduce that
\begin{align*}
  F
  &\leq C_B \int_0^t\int_\Omega \sum_{i=1}^{+\infty} (i c_i)^{p - 1}\sum_{j= 1}^{+\infty}\frac{i^{3 + \eta - p}(i + j)}{(i + j)^{1 + \eta}} c_{i+j}\\
  &\leq C_B \int_0^t\int_\Omega \sum_{i=1}^{+\infty} (i c_i)^{p - 1}\sum_{j= 1}^{+\infty}(i+j)^{3 - p} c_{i+j}.
\end{align*}

Since $p \geq 2$, one has $\sum_{j= 1}^{+\infty}(i+j)^{3 - p} c_{i+j} \leq \rho_1$, furthermore we can write $\sum_{i=1}^{+\infty} (i c_i)^{p - 1} \leq \left(\max_{j \in \N^*} jc_j\right)^{p - 2}\sum_{i=1}^{+\infty} i c_i \leq \rho_1^{p - 2}\rho_1 = \rho_1^{p - 1}$.

Hence, we finally obtain
\begin{align*}
  F \leq C_B \int_0^t\int_\Omega\rho_1^{p} &\leq C_\Omega \left(1 + \int_0^t\int_\Omega\rho_1^{p + 1}\right),
\end{align*}
where $C_\Omega$ is a constant that depends on $T$, $\Omega$, $p$ and $C_B$.

We now treat the coagulation part. Using \ref{ass:bounda'}, we obtain
\begin{equation*}
  i^{2 + \eta} a_{i-j, j} \leq C_a(i-j)j.
\end{equation*}
We recall the notation $C_k := (i^kc_i)_{i\in\N^*}$ extended by $0$ for $i\leq 0$ and we denote by $*$ the convolution between two sequences. We deduce the bound
\begin{align*}
  K
  &\leq \frac{C_a}{2} \int_0^t\int_\Omega \sum_{i=1}^{+\infty} (i c_i)^{p - 1}i^{2 - p}\sum_{j=1}^{i - 1} (i - j)c_{i-j}jc_j\\
  &\leq \frac{C_a}{2}\int_0^t\int_\Omega \sum_{i=1}^{+\infty} (ic_i)^{p - 1} C_1*C_1(i).
\end{align*}
To obtain the last inequality, we used $p\geq 2$ and hence $i^{2 - p} \leq 1$.
We can now obtain the bound
\begin{align*}
  K
  &\leq \frac{C_a}{2}\int_0^t\int_\Omega \|C_1\|_{l^{\infty}}^{p - 1} \|C_1*C_1\|_{l^1}\\
  &\leq \frac{C_a}{2}\int_0^t\int_\Omega \|C_1\|_{l^{\infty}}^{p - 1}\|C_1\|_{l^1}^2\\
  &\leq \frac{C_a}{2}\int_0^t\int_\Omega \|C_1\|_{l^1}^{p + 1}\\
  &= \frac{C_a}{2}\int_0^t\int_\Omega \rho_1^{p + 1}.
\end{align*}

Combining these estimates, we deduce that for $K_1:= C_\Omega + \frac{C_a}{2}$, one has
\begin{align}\label{eq:energy1}
    &\phantom{=}\frac1p\int_\Omega \sum_{i=1}^{+\infty} (i^{\frac{3 + \eta}{p}} c_i(t))^p
    + \underline{D}\frac{4(p - 1)}{p^2}\int_0^t\int_\Omega \sum_{i=1}^{+\infty} |\nabla i^{\frac{3 + \eta}{2}} c_i^\frac{p}{2}|^2\\
    &\leq K_1\left( 1 +  \int_0^t\int_\Omega \rho_1^{p + 1}\right)
    + \frac1p\int_\Omega \sum_{i=1}^{+\infty} (i^{\frac{3 + \eta}{p}} c_{i, \init})^p.\nonumber
  \end{align}

  Using the bounds $2\leq p\leq 3$, we note that
  \begin{align*}
    \int_\Omega \sum_{i=1}^{+\infty} (i^{\frac{3 + \eta}{p}} c_{i, \init})^p
    &= \|\|i^{\frac{3 + \eta}{p}} c_{i, \init}\|_{l^p}\|_{\fsL^p(\Omega)}^p\\
    &\leq \|\|i^{\frac{3 + \eta}{p}} c_{i, \init}\|_{l^1}\|_{\fsL^p(\Omega)}^p\\
    &\leq C'_\Omega \|\rho_{\frac{3 + {\eta}}{2}, \init}\|_{\fsL^3(\Omega)}^p,
  \end{align*}
  where $C'_\Omega$ is a constant that depends on $|\Omega|$ and $p$.\medskip

  Thus, one gets for $C_p := \max(\frac{p^2}{4(p - 1)\underline D}, p)\left(K_1 + \frac{C'_\Omega}{p} \|\rho_{\frac{3 + {\eta}}{2}, \init}\|_{\fsL^3(\Omega)}^p\right)$ the bound
  \begin{align}\label{eq:coagulation-energyp}
    &\phantom{=}\|\|i^{\frac{3 + \eta}{2}} c_i^{\frac{p}{2}}\|_{l^2}\|_{\fsL^\infty([0, T], \fsL^2(\Omega))}^2+
    \|\|\nabla(i^{\frac{3 + \eta}{2}} c_i^{\frac{p}{2}})\|_{l^2}\|_{\fsL^2(\Omega_T)}^2\nonumber\\
    &\leq C_p \left(\int_0^T\int_\Omega \rho_1^{p+1} + 1\right).
  \end{align}

  We first consider the case $p= 2$. We remark that
  \begin{align*}
    \int_0^T\int_\Omega |\nabla \rho_1|^2
    &\leq \int_0^T\int_\Omega \left(\sum_{i=1}^{+\infty} |\nabla ic_i|\right)^2\\
    &= \|\|\nabla ic_i\|_{l^1} \|_{\fsL^2(\Omega_T)}^{2}\\
    &\leq \|i^{-\frac{1 + \eta}{2}}\|_{l^2}^2\|\|\nabla i^{\frac{3+\eta}{2}} c_i\|_{l^2} \|_{\fsL^{2}(\Omega_T)}^{2}.
  \end{align*}
  Since $\frac{1 + \eta}{2} > \frac{1}{2}$ we have $i^{-\frac{1 + \eta}{2}}\in l^2$. Together with \eqref{eq:coagulation-energyp}, we obtain for some constant $C > 0$,
  \begin{equation}\label{eq:coagulation:energy}
    \int_0^T\int_\Omega |\nabla \rho_1|^2 \leq C\left(1 +\int_0^T\int_\Omega \rho_1^3\right).
  \end{equation}

  We now consider the case $p> 2$, using a Sobolev embedding we have for \mbox{$2^* := \frac{2d}{d - 2}$} (and any $q < +\infty$ if $d\leq 2$) and a constant $C_\Omega$,
  \begin{align*}
    \|\|i^{\frac{3 + \eta}{2}} c_i^{\frac{p}{2}}\|_{\fsL^2([0, T], \fsL^{2^*}(\Omega))}\|_{l^2}^2\leq
    C_\Omega\|\|i^{\frac{3 + \eta}{2}} c_i^{\frac{p}{2}}\|_{\fsL^2([0, T], \fsH^1(\Omega))}\|_{l^2}^2.
  \end{align*}

  We now use Minkowski's inequality to obtain
  \begin{align*}
    \|\|i^{\frac{3 + \eta}{p}} c_i\|_{l^p}\|_{\fsL^p([0, T], \fsL^{\frac{2^*p}{2}}(\Omega))}^p
    &= \|\|\left(i^{\frac{3 + \eta}{p}} c_i\right)^{\frac{p}{2}}\|_{l^2}\|_{\fsL^2([0, T], \fsL^{2^*}(\Omega))}^2\\
    &
    \leq
    \|\|i^{\frac{3 + \eta}{2}} c_i^{\frac{p}{2}}\|_{\fsL^2([0, T], \fsL^{2^*}(\Omega))}\|_{l^2}^2.
  \end{align*}

  We can use Hölder's inequality to deduce that
    \begin{align*}
      \|\rho_1\|_{\fsL^p([0, T], \fsL^{\frac{2^*p}{2}}(\Omega))}^p
      &= \|\|i c_i\|_{l^1}\|_{\fsL^p([0, T], \fsL^{\frac{2^*p}{2}}(\Omega))}^p\\
    &\leq
    \|\|i^{1-\frac{3 + \eta}{p}}\|_{l^{p'}}\|i^{\frac{3 + \eta}{p}} c_i\|_{l^p}\|_{\fsL^p([0, T], \fsL^{\frac{2^*p}{2}}(\Omega))}^p.
  \end{align*}
  We note that $\|i^{1-\frac{3 + \eta}{p}}\|_{l^{p'}}$ is finite provided that $\left(\frac{3 + \eta}{p} - 1\right)p' > 1$. This condition amounts to
  \begin{equation}\label{eq:cdtp}
    p < 2 + \frac{\eta}{2}.
  \end{equation}
  Collecting the preceding estimates, we have thus shown that under the condition \eqref{eq:cdtp}, one has for some constant $C > 0$,
  \begin{equation}\label{eq:coagulation:energyp}
    \|\rho_1\|_{\fsL^p([0, T], \fsL^{\frac{2^*p}{2}}(\Omega))}^p\leq C\left(1 +\int_0^T\int_\Omega \rho_1^{p+1}\right).
  \end{equation}
  This estimate provides slightly more integrability on the left-hand side than \eqref{eq:coagulation:energy}.

    We have thus shown two energy estimates. We point out that to obtain these estimates, we use the strong assumptions \ref{ass:bounda'} and \ref{ass:boundB'} to convert some $l^2$ (or $l^p$) norms to $l^1$ ones. We now prove an a priori bound using \cref{thm:main-neumann}. \medskip

    Define $w(t, \cdot) := \int_0^t \left(\sum_{i=1}^{+\infty} d_ii c_{i}\right)(s, \cdot)\dd s$. We show that $w$ satisfies a suitable parabolic equation.

  We note that
  \begin{equation*}
    \partial_t w = \sum_{i=1}^{+\infty} d_ii c_{i},
  \end{equation*}
  whereas by \eqref{eq:coagulation}
  \begin{equation*}
    \Delta w = \int_0^t \left(\sum_{i=1}^{+\infty} d_ii \Delta c_{i} \right)= \int_0^t \sum_{i=1}^{+\infty}\left( i \partial_t c_{i} - iQ_i - iF_i\right).
  \end{equation*}  

  Using \eqref{eq:weak} and \eqref{eq:weakF} (and \ref{ass:massB}), we see that $\sum_{i=1}^{+\infty} iQ_i = \sum_{i=1}^{+\infty} iF_i = 0$. Hence, we obtain that
    \begin{equation*}
    \Delta w =  \sum_{i=1}^{+\infty} i \left(c_{i}(t) - c_{i, \init}\right) = \rho_1(t) - \rho_{1, \init}.
  \end{equation*}  

  Thus, we have shown that

  \begin{align*}
    \partial_t w
    &= \sum_{i=1}^{+\infty} d_i ic_{i}
    = \mu \sum_{i=1}^{+\infty} ic_i
    = \mu \left(\Delta w + \rho_{1, \init}\right),
  \end{align*}
  with $\mu := \frac{\sum_{i=1}^{+\infty} d_i i c_i}{\sum_{i=1}^{+\infty} i c_i}$. We note that $\underline D\leq \mu \leq \overline{D}$ thanks to \ref{ass:boundd}. Furthermore, $w$ satisfies homogeneous Neumann boundary conditions $\nabla w \cdot n = 0$ on $[0, T]\times\partial\Omega$.\medskip

  We note that $w_\init = 0$, $w\geq 0$, $\partial_t w \geq 0$ and $\mu\rho_{1, \init}$ is bounded in $\fsL^\infty(\Omega_T)$. We can thus apply \cref{thm:main-neumann} to deduce that for some $\alpha > 0$ and for some $C_\alpha > 0$ with the desired dependence we have 
  \begin{equation}\label{eq:holder1}
    \|w\|_{\fsC^\alpha(\overline\Omega_T)} \leq C_\alpha.
  \end{equation}

  This a priori estimate is sufficient to complete the proof of \cref{thm:coagulation-basecase} thanks to \cref{thm:interpolation}. \medskip
  
  We recall that $\rho_1 = \Delta w + \rho_{1, \init} \leq \Delta w'$, with $w' = w + \frac{|x|^2}{2d} \|\rho_{1, \init}\|_{\fsL^\infty(\Omega)}$.
  Since $w'$ enjoys the same regularity as $w$ we can apply \cref{thm:interpolation} with $\alpha$ found in \eqref{eq:holder1}. Hence, with the notations of this Lemma one has for $0 < t < T$
  \begin{equation*}
    \|\rho_1(t) \|_{\fsL^{3+\eps_1}(\Omega)}
    \le  C_{d, \alpha}\, \| {w'}(t) \|_{\fsC^{\alpha}(\overline{\Omega}) }^{\frac{1}{3- \alpha}}\,
    \| \nabla \rho_1(t)\|_{\fsL^2(\Omega)}^{\frac{2}{3+\eps_1}}
    +  C_{d,\alpha}\, \| {w'}(t) \|_{\fsC^{\alpha}(\overline{\Omega})}.
  \end{equation*}

  Integrating this estimate with respect to time, we deduce that for some constant $C$ with the same dependence as in the statement of \cref{thm:coagulation-basecase}
  \begin{equation*}
    \|\rho_1 \|_{\fsL^{3+\eps_1}(\Omega_T)}^{3+\eps_1}
    \le  C
    \| \nabla \rho_1\|_{\fsL^2(\Omega_T)}^{2}
    +  C.
  \end{equation*}

  Using the energy estimate \eqref{eq:coagulation:energy} (and allowing $C$ to change from line to line), we obtain
  \begin{equation*}
    \|\rho_1 \|_{\fsL^{3+\eps_1}(\Omega_T)}^{3+\eps_1}
    \le  C
    \| \rho_1\|_{\fsL^3(\Omega_T)}^{3}
    +  C.
  \end{equation*}

  Young's inequality then gives that $\|\rho_1 \|_{\fsL^{3+\eps_1}(\Omega_T)} \leq C$ for some $\eps_1 > 0$ (and hence, we deduce the first inequality of \cref{thm:coagulation-basecase} by taking \mbox{$\eps := \eps_1$}). For the second inequality, we choose $p \in ]2, \min(3, 2 + \frac{\eta}{2}, 2 + \eps_1)[$ and we apply \eqref{eq:coagulation:energyp} with $2^*\frac{p}{2} = 2^* + \eps'$.
  \end{proof}

  We can now improve this estimate to higher moments.
\begin{prop}\label{thm:coagulation-recursion}
    We consider the framework of \cref{thm:maind} and some $T > 0$. Then, for $k\in\N^*$ there exist $\eps_k, \, \eps'_k > 0$ and a constant $C$ depending on the parameters of the problem, $k$, $\Omega$, $d$, $T$ as well as on $\|\rho_{l + \frac12 + \frac\eta2, \init}\|_{\fsL^{3}(\Omega)}$ and $\|\rho_{l, \init}\|_{\fsL^{\infty}(\Omega)}$ for $l\leq k$, such that
    $$\|\rho_k\|_{\fsL^{3+\eps_k}(\Omega_T)} + \|\rho_k\|_{\fsL^2([0, T], \fsL^{2^*+\eps'_k}(\Omega))} \leq C_k,$$
    where $2^*$ is defined in \cref{thm:coagulation-basecase}.
\end{prop}

\begin{proof}
  We prove the claim by induction on $k$, the base case $k=1$ being provided by Proposition~\ref{thm:coagulation-basecase}. Let $k>1$ be an integer such that the claim holds for $k - 1$.
  
  Let $m$ be the solution to the following parabolic equation
  \begin{equation}\label{eq:m}
    \left\{
      \begin{aligned}
        &\partial_t m - \overline D \Delta m = -\sum_{i=1}^{+\infty} i^k F_i(c) \quad
        &&\text{in }\Omega_T,\\
        &\nabla m\cdot n = 0\quad
        &&\text{on }[0, T]\times\partial\Omega,\\
        &m(0) = 0\quad
        &&\text{in }\Omega.
      \end{aligned}
    \right. 
  \end{equation}
  Using the weak formulation \eqref{eq:weakF}, we obtain the identity
  \begin{align*}
    \sum_{i=1}^{+\infty} i^k\,F_i(c) = \sum_{i=2}^{+\infty}B_ic_i\left(\sum_{j = 1}^{i - 1}\beta_{i, j}j^k - i^k\right).
  \end{align*}
  Using \ref{ass:massB}, we note that
  $\sum_{j = 1}^{i - 1}\beta_{i, j}j^k - i^k\leq i^{k - 1}\left(\sum_{j = 1}^{i - 1}\beta_{i, j}j - i\right) = 0$. Thus, $\sum_{i=1}^{+\infty} i^k\,F_i(c) \leq 0$. In particular by the parabolic minimum principle, one has $m \geq 0$. The quantity $m$ has been introduced in such a way that we will be able to work mostly with non-negative functions.\medskip

  We can sum \eqref{eq:coagulation} over $i$ and add \eqref{eq:m} to obtain
  \begin{equation*}
    \partial_t \left(\rho_{k} + m \right) - \Delta\left(\mu_k (\rho_{k} + m)\right) = \sum_{i=1}^{+\infty} i^kQ_{i},
  \end{equation*}
  with $\mu_k = \frac{\sum_{i=1}^{+\infty} d_i i^k c_i + \overline D m}{\sum_{i=1}^{+\infty} i^k c_i + m}$. We notice again that $\underline{D} \leq \mu_k \leq \overline{D}$.

  Define $w(t, \cdot) := \int_0^t \left(\mu_k[\rho_k + m]\right)(s, \cdot)\dd s$. Proceeding as in \cref{thm:coagulation-basecase}, we obtain
  \begin{equation*}
    \partial_t w = \mu_k(\rho_k + m),
  \end{equation*}
  and
  \begin{align*}
    \Delta w
    &= \int_0^t \Delta\left(\mu_k[\rho_k + m]\right)\\
    &= \int_0^t \left(\partial_t (\rho_k + m) - \sum_{i=1}^{+\infty} i^kQ_i\right)\\
    &= \rho_k(t) + m(t) - \rho_{k, \init} - \int_0^t \sum_{i=1}^{+\infty} i^kQ_i.
  \end{align*}

  It follows that
  \begin{equation}\label{eq:eqwk}
    \partial_t w = \mu_k\Delta w + \mu_k\left(\rho_{k, \init} + \int_0^t \sum_{i=1}^{+\infty} i^kQ_i\right).
  \end{equation}
   
  We now bound the forcing terms. The weak formulation \eqref{eq:weak} gives
  \begin{align*}
    2\sum_{i=1}^{+\infty} i^kQ_i
    &= \sum_{i=1}^{+\infty}\sum_{j  = 1}^{+\infty} a_{i, j}c_ic_j\left((i+j)^k - i^k - j^k\right)\\
    &= \sum_{i=1}^{+\infty}\sum_{j  = 1}^{+\infty} a_{i, j}c_ic_j\left(\sum_{l = 1}^{k - 1} \binom{k}{l}i^lj^{k-l}\right).
  \end{align*}

  We note that \ref{ass:bounda'} implies
  \begin{equation*}
    a_{i, j} \leq C_a \frac{\sqrt{ij}^2}{(i+j)^{2+{\eta}}}\leq\frac{C_a}{2} \frac{(i + j)^2}{(i+j)^{2+{\eta}}} \leq \frac{C_a}{2}.
  \end{equation*}

  Hence, for some constant $C > 0$ with the same dependence as in the statement and that may change from line to line
  \begin{align}\label{eq:boundQ}
   0\leq \sum_{i=1}^{+\infty} i^kQ_i
    & \leq C\sum_{l = 1}^{k - 1}\binom{k}{l}\sum_{i=1}^{+\infty}\sum_{j  = 1}^{+\infty} (i^lc_i)(j^{k-l}c_j)\nonumber\\
    & \leq C \sum_{l = 1}^{k - 1} \rho_l\rho_{k-l}.
  \end{align}

  We now introduce the notation $p+$ for some $p\in(1, +\infty)$. This denotes a constant $p + \kappa$, where $\kappa > 0$ depends on $d$ and on $\eps_l, \eps'_l$ for $l < k$ (we recall that the $\eps_l$, $\eps'_l$ are given by the induction procedure).
    
  By the induction hypothesis, for $0 < l < k$, $\rho_l$ is bounded in $\fsL^2([0, T],\fsL^{2^*+}(\Omega_T))$. Hence, $\rho_l\rho_{k - l}$ is bounded in $\fsL^1([0, T], \fsL^{\frac{2^*+}{2}}(\Omega))$. It follows that the quantity
  $$\mu_k\left(\rho_{k, \init} + \int_0^t \sum_{i=1}^{+\infty} i^kQ_i\right)$$
  is bounded in $\fsL^\infty([0, T], \fsL^{\frac{2^*+}{2}}(\Omega))$. Noticing that
  $2 - \frac{2}{+\infty} - \frac{2d}{2^*+} = (4 - d)+ > 0$, this is an admissible forcing for \cref{thm:main-neumann}.  We recall that $w$ satisfies \eqref{eq:eqwk} together with homogeneous Neumann boundary conditions. Furthermore, $\partial_t w \geq 0$, $w_\init = 0$ and $w \geq 0$, thus, we deduce from \cref{thm:main-neumann} that for some $\alpha, \, C_\alpha > 0$, we have $\|w\|_{\fsC^{\alpha}(\overline\Omega_T)} \leq C_\alpha$.
  
  Furthermore $\Delta w = \rho_k + m - \rho_{k, \init} - \int_0^t\sum_{i = 1}^{+\infty}i^kQ_i$. Let $\phi$ be the solution to
  \begin{equation*}\left\{
    \begin{aligned}
      &\Delta \phi(t, \cdot) = \rho_{k, \init} + \int_0^t\sum_{i = 1}^{+\infty}i^kQ_i
      &&\quad\text{in }\Omega\times[0, T],\\
      &\phi(t, \cdot) = 0
      &&\quad\text{on }\partial\Omega\times[0, T].
    \end{aligned}\right.
  \end{equation*}
  Since the right-hand side belongs to $\fsL^\infty([0, T], \fsL^{\frac{2^*}{2}+}(\Omega))$, in particular, it belongs in dimension $d\leq 4$ to $\fsL^\infty([0, T], \fsL^{\frac{d}{2}+}(\Omega))$. Classical elliptic regularity then shows that $\phi$ is bounded in $\fsL^\infty([0, T], \fsC^{\beta}(\overline\Omega))$ for some $\beta >0$.
  
  Recalling that $m\geq 0$, and setting $w' := w + \phi$, we have thus shown that there exist \mbox{$\alpha, \, C_\alpha > 0$} such that
  \begin{equation}\label{eq:apriori}
    0 \leq \rho_k \leq \Delta w'\quad\text{and }\| w'\|_{\fsL^\infty([0, T], \fsC^\alpha(\overline\Omega))} \leq C_\alpha.
  \end{equation}
  We now derive two energy inequalities in the spirit of \cref{thm:coagulation-basecase}.\medskip

  Let $p\in [2, \min(3, k + 1 + \frac{\eta}{2})[$. We multiply \eqref{eq:coagulation} by $i^{2k + 1 + \eta}c_i^{p - 1}$, integrate over $\Omega_t$ and sum over $i$ to obtain for $0 < t < T$
  \begin{align}\label{eq:boundNRJ}
    &\frac1p \int_\Omega \sum_{i=1}^{+\infty} i^{2k + 1 + \eta} c_i^p(t) - \frac1p\int_\Omega \sum_{i=1}^{+\infty} i^{2k + 1 +\eta}c_{i, \init}^{p} + \frac{4(p - 1)}{p^2}\underline D\int_0^t\int_\Omega \sum_{i=1}^{+\infty} |\nabla i^{k + \frac 12 + \frac{{\eta}}{2}}  c_i^{\frac{p}{2}}|^2\nonumber\\
    &\leq \int_0^t\int_\Omega \sum_{i=1}^{+\infty}i^{2k +1 + \eta}  c_i^{p - 1}  (Q_i^+ + F_i^+).
  \end{align}

  We first bound the contribution of the coagulation term. We note that \ref{ass:bounda'} implies
  \begin{align}\label{eq:aij}
    a_{i, j}\nonumber
    &\leq C_a \frac{ij}{(i+j)^{2 +{\eta}}}\nonumber\\
    &\leq C_a \frac{i^kj^k}{(i+j)^{k + 1 + {\eta}}}\frac{(i+j)^{k - 1}}{(ij)^{k-1}}\nonumber\\
    &\leq 2^{k - 1} C_a \frac{i^kj^k}{(i+j)^{k + 1 + {\eta}}},
  \end{align}
  where we used to obtain the last line that for $i, j \in\N^*$ we have $i + j \leq 2ij$.
  
  We obtain the bound
\begin{align*}
    K
    &:= \int_0^t\int_\Omega \sum_{i=1}^{+\infty}i^{2k +1 + \eta}  c_i^{p - 1}Q_i^+\\
    &=\frac12\int_0^t\int_\Omega \sum_{i=1}^{+\infty} (i^k c_i)^{p - 1}\sum_{j=1}^{i - 1}i^{k + 1 + {\eta} + k (2- p)}  a_{i-j, j}c_{i-j}c_j\\
    &\leq  2^{k - 2} C_a \int_0^t\int_\Omega \sum_{i=1}^{+\infty} (i^k c_i)^{p - 1}\sum_{j=1}^{i - 1}[(i -j)^k c_{i-j}][j^kc_j].
\end{align*}
The last inequality follows from \eqref{eq:aij} and the bound $i^{k(2 - p)} \leq 1$ (since $p \geq 2$). We deduce
\begin{align}\label{eq:boundK}
    K  
    &\leq  2^{k - 2} C_a\int_0^t\int_\Omega \sum_{i=1}^{+\infty} (i^kc_i)^{p - 1} C_k*C_k(i)\nonumber\\
    &\leq  2^{k - 2} C_a\int_0^t\int_\Omega \|C_k\|^{p - 1}_{l^{\infty}} \|C_k*C_k\|_{l^1}\nonumber\\
    &\leq  2^{k - 2} C_a\int_0^t\int_\Omega \|C_k\|^{p - 1}_{l^{\infty}}\|C_k\|_{l^1}^2\nonumber\\
    &\leq  2^{k - 2} C_a\int_0^t\int_\Omega \rho_k^{p+1}.
  \end{align}

  We now bound the contribution of the fragmentation term, we use both the bound $\beta_{i + j, i} \leq \frac{i + j}{i}$ and \ref{ass:boundB'}
    \begin{align}
    F
    &:= \int_0^t\int_\Omega \sum_{i=1}^{+\infty}i^{2k +1 + \eta}  c_i^{p - 1} F_i^+\nonumber\\
    &= \int_0^t\int_\Omega \sum_{i=1}^{+\infty}i^{2k +1 + \eta}  c_i^{p - 1}\sum_{j = 1}^{+\infty}B_{i + j}\beta_{i + j, i}c_{i+j}\nonumber\\
    &\leq C_B \int_0^t\int_\Omega \sum_{i=1}^{+\infty} i^{k(2 - p)}(i^k c_i)^{p - 1}\sum_{j=1}^{+\infty}i^{k}\frac{i^{1 + \eta}}{i(i+j)^\eta}c_{i+j}\nonumber\\
    &\leq C_B\int_0^t\int_\Omega \sum_{i=1}^{+\infty} (i^kc_i)^{p - 1} \sum_{j=1}^{+\infty}(i+j)^{k}c_{i + j}\nonumber.
  \end{align}
  
  We notice that $\sum_{j=1}^{+\infty}(i+j)^{k}c_{i + j} \leq \rho_k$ and
  \begin{align*}
    \sum_{i=1}^{+\infty} (i^kc_i)^{p - 1}
    &\leq \left(\max_{j\in\N^*}i^kc_i\right)^{p - 2} \sum_{i=1}^{+\infty} i^kc_i\\
    &\leq \rho_k^{p - 1}.
  \end{align*}

  Hence, we deduce that
  \begin{align}\label{eq:boundF}
    F
    &\leq C_B\int_0^t\int_\Omega \rho_k^p\nonumber\\
    &\leq C\left(1 + \int_0^t\int_\Omega \rho_k^{p +1}\right),
  \end{align}
  where $C$ is a constant depending on $C_B$, $\Omega$, $T$.

  Using the bounds $2\leq p\leq 3$, we note that
  \begin{align}\label{eq:bound-init}
    \int_\Omega \sum_{i=1}^{+\infty} i^{2k + 1 + \eta} c_{i, \init}^p
    &= \|\|i^{\frac{2k + 1 + \eta}{p}} c_{i, \init}\|_{l^p}\|_{\fsL^p(\Omega)}^p\nonumber \\
    &\leq \|\|i^{\frac{2k + 1 + \eta}{p}} c_{i, \init}\|_{l^1}\|_{\fsL^p(\Omega)}^p\nonumber\\
    &\leq C'_\Omega \|\rho_{k + \frac{1 + {\eta}}{2}, \init}\|_{\fsL^3(\Omega)}^p,
  \end{align}
  where $C'_\Omega$ is a constant that depends on $|\Omega|$ and $p$.\medskip

  Using \eqref{eq:boundNRJ}, \eqref{eq:boundK}, \eqref{eq:boundF} and \eqref{eq:bound-init}, we deduce that for some $C_p$ with the same dependence as in statement (and also depending on $p$), one has
  \begin{align}\label{eq:NRJintk}
    \sup_{0\leq t \leq T}\int_\Omega \sum_{i=1}^{+\infty} |i^{k + \frac 12 + \frac{{\eta}}{2}}  c_i^{\frac{p}{2}}|^2 + \int_0^T\int_\Omega \sum_{i=1}^{+\infty} |\nabla i^{k + \frac 12 + \frac{{\eta}}{2}}  c_i^{\frac{p}{2}}|^2
    \leq C_p\left(\int_0^T\int_\Omega \rho_k^{p+1} + 1\right).
  \end{align}

  We first consider the case $p=2$, we can compute
  \begin{align}\label{eq:boundconv}
    \int_0^T\int_\Omega |\nabla \rho_k|^2
    &\leq \|\|\nabla i^kc_i\|_{l^1} \|_{\fsL^{2}(\Omega_T)}^2\nonumber\\
    &\leq \| i^{-\frac{1+\eta}{2}}\|_{l^2} \|\|\nabla i^{k + \frac{1 + {\eta}}{2}} c_i\|_{l^2} \|_{\fsL^{2}(\Omega_T)}^{2},
  \end{align}
  where we used the discrete Cauchy--Schwarz inequality to obtain the last line.

  Using \eqref{eq:NRJintk}, we obtain for a constant $C_2 > 0$
  \begin{align}\label{eq:coagulation:energy2k}
    \int_0^T\int_\Omega |\nabla \rho_k|^2
    \leq C_2\left(\int_0^T\int_\Omega \rho_k^{3} + 1\right).
  \end{align}

  We now consider the case $p> 2$. Using a Sobolev embedding we have for \mbox{$2^* := \frac{2d}{d - 2}$} (and any $q < +\infty$ if $d\leq 2$) and a constant $C_\Omega$ depending only on $\Omega$,
  \begin{align*}
    \|\|i^{\frac{2k + 1 + \eta}{2}} c_i^{\frac{p}{2}}\|_{\fsL^2([0, T], \fsL^{2^*}(\Omega))}\|_{l^2}^2
    &\leq C_\Omega\|\|i^{\frac{2k + 1 + \eta}{2}} c_i^{\frac{p}{2}}\|_{\fsL^2([0, T], \fsH^1(\Omega))}\|_{l^2}^2\\
    &=C_\Omega \int_0^T\int_\Omega \sum_{i=1}^{+\infty} \left(|\nabla(i^{\frac{2k + 1 + \eta}{2}} c_i^{\frac{p}{2}})|^2 + |i^{\frac{2k + 1 + \eta}{2}} c_i^{\frac{p}{2}}|^2\right).
  \end{align*}

  We now use Minkowski's inequality to interchange the $\fsL^{2^*}(\Omega)$ and the $l^2$ norms, obtaining
  \begin{align*}
    \|\|i^{\frac{2k + 1 + \eta}{p}} c_i\|_{l^p}\|_{\fsL^p([0, T], \fsL^{\frac{2^*p}{2}}(\Omega))}^p
    &= \|\|\left(i^{\frac{2k + 1 + \eta}{p}} c_i\right)^{\frac{p}{2}}\|_{l^2}\|_{\fsL^2([0, T], \fsL^{2^*}(\Omega))}^2\\
    &
    \leq
    \|\|i^{\frac{2k + 1 + \eta}{2}} c_i^{\frac{p}{2}}\|_{\fsL^2([0, T], \fsL^{2^*}(\Omega))}\|_{l^2}^2.
  \end{align*}

  We can use Hölder's inequality to deduce that
    \begin{align*}
      \|\rho_k\|_{\fsL^p([0, T], \fsL^{\frac{2^*p}{2}}(\Omega))}^p
      &= \|\|i^k c_i\|_{l^1}\|_{\fsL^p([0, T], \fsL^{\frac{2^*p}{2}}(\Omega))}^p\\
    &\leq
    \|\|i^{k-\frac{2k + 1 + \eta}{p}}\|_{l^{p'}}\|i^{\frac{2k + 1 +\eta}{p}} c_i\|_{l^p}\|_{\fsL^p([0, T], \fsL^{\frac{2^*p}{2}}(\Omega))}^p.
  \end{align*}
  We note that $\|i^{k-\frac{2k + 1 + \eta}{p}}\|_{l^{p'}}$ is finite provided that $\left(\frac{2k + 1 + \eta}{p} - 1\right)p' > 1$. This condition amounts to
  \begin{equation}\label{eq:cdtpk}
    p < k + 1 + \frac{\eta}{2}.
  \end{equation}
  Collecting the preceding estimates, we have thus shown that under the condition \eqref{eq:cdtpk} one has for some constant $C > 0$,
  \begin{equation}\label{eq:coagulation:energypk}
    \|\rho_k\|_{\fsL^p([0, T], \fsL^{\frac{2^*p}{2}}(\Omega))}^p\leq C\left(1 +\int_0^T\int_\Omega \rho_k^{p+1}\right).
  \end{equation}

  We are now in the same situation as in the proof of \cref{thm:coagulation-basecase} and we conclude with an interpolation argument.

  One can apply \cref{thm:interpolation} with $\alpha$ found in \eqref{eq:apriori}. Hence, there exists some $\eps_k$ such that one has for $0 < t < T$
  \begin{equation*}
    \|\rho_k(t) \|_{\fsL^{3+\eps_k}(\Omega)}
    \le  C_{d, \alpha}\, \| {w'}(t) \|_{\fsC^{\alpha}(\overline{\Omega}) }^{\frac{1}{3- \alpha}}\,
    \| \nabla \rho_k(t)\|_{\fsL^2(\Omega)}^{\frac{2}{3+\eps_k}}
    +  C_{d,\alpha}\, \| {w'}(t) \|_{\fsC^{\alpha}(\overline{\Omega})}.
  \end{equation*}

  Integrating this estimate, we deduce that for some constant $C$ with the same dependence as in the statement of \cref{thm:coagulation-basecase}
  \begin{equation*}
    \|\rho_k \|_{\fsL^{3+\eps_k}(\Omega_T)}^{3+\eps_k}
    \le  C
    \| \nabla \rho_k\|_{\fsL^2(\Omega_T)}^{2}
    +  C.
  \end{equation*}

  Using the energy estimate \eqref{eq:coagulation:energy2k} (and allowing $C$ to change from line to line), we obtain
  \begin{equation*}
    \|\rho_k \|_{\fsL^{3+\eps_k}(\Omega_T)}^{3+\eps_k}
    \le  C
    \| \rho_k\|_{\fsL^3(\Omega_T)}^{3}
    +  C.
  \end{equation*}

  Young's inequality then gives that $\|\rho_k \|_{\fsL^{3+\eps_k}(\Omega_T)} \leq C$ for some $\eps_k > 0$. We deduce the second bound thanks to the estimate \eqref{eq:coagulation:energypk}.

  \end{proof}
  
  \begin{proof}[Proof of Theorem~\ref{thm:maind}]
    Similarly to \cref{sec:endproof}, we begin by showing that all the moments are bounded in $\fsL^\infty(\Omega_T)$. For any $k\in\N^*$, we define the sequence $(q^k_{r})_{r\in\N^*}$ by $\frac{1}{q^k_{r + 1}} = \frac{2}{q^k_r} - \frac13$ and $q^k_1 = 3 + \eps_k$. Since $q^k_1 > 3$, this sequence becomes non-positive in a finite number of steps. Let $k\in\N^*$, we note that
    \begin{equation*}
      \left(\partial_t - d_i\Delta\right)i^{k + 1 + \eta}c_i \leq i^{k + 1 + \eta}\left(Q_i^+ + F_i^+\right). 
    \end{equation*}
    Using \ref{ass:bounda'}, we can now bound
    \begin{align*}
      i^{k + 1 + \eta}Q_i^+
      &= \frac12 i^{k + 1 + \eta}\sum_{j = 1}^{i - 1} a_{i - j, j}c_{i - j}c_j\\
      &\leq \frac{C_a}{2}i^{k - 1}\sum_{j = 1}^{i - 1} (i - j)c_{i - j}jc_j.
    \end{align*}
      
    We use the same decomposition as in the proof of \cref{thm:main2}:
    \begin{align*}
      i^{k - 1}\sum_{j = 1}^{i - 1} (i - j)c_{i - j}jc_j
      &= i^{k - 1}\sum_{j = 1}^{\lfloor \frac{i}{2} \rfloor} (i - j)c_{i - j}jc_j
      + i^{k - 1}\sum_{j = \lfloor  \frac{i}{2} \rfloor + 1}^{i - 1} (i - j)c_{i - j}jc_j\\
      &\leq 2^{k - 1}\sum_{j = 1}^{\lfloor  \frac{i}{2} \rfloor} (i - j)^kc_{i - j}jc_j
      + 2^{k - 1}\sum_{j = \lfloor  \frac{i}{2} \rfloor + 1}^{i - 1} (i - j)c_{i - j}j^kc_j.
    \end{align*}

    To obtain the last line, we used that $i \leq 2(i - j)$ if $1\leq j\leq \left\lfloor  \frac{i}{2} \right\rfloor$ and $i\leq 2j$ if $\left\lfloor  \frac{i}{2} \right\rfloor + 1\leq j \leq i - 1$. Proceeding as in the proof of \cref{thm:main2}, we deduce 
    $$i^{k + 1 + \eta}Q_i^+ \leq C_a 2^{k - 1}\rho_k\rho_1.$$

    Similarly, (using again $\beta_{i+j, i}\leq \frac{i + j}{i}$)
    \begin{align*}
      i^{k + 1 + \eta}F_i^+
      & = i^{k + 1 + \eta}\sum_{j = 1}^{+\infty} B_{i + j}\beta_{i + j, i}c_{i + j}\\
      &\leq C_B \sum_{j = 1}^{+\infty} \frac{i^{k + 1 + \eta}}{i(i+j)^\eta}c_{i + j}\\
      &\leq C_B\rho_k.
    \end{align*}

    Hence,
    \begin{equation}\label{eq:inductionci}
      \left(\partial_t - d_i\Delta\right)i^{k + 1 + \eta}c_i \leq C_a2^{k - 1}\rho_k \rho_1 + C_B\rho_k \leq C_a2^{k - 1}\rho_k^2 + C_B\rho_k. 
    \end{equation}

    We note that the solution to the Neumann heat equation in dimension $d\leq 4$ with a forcing in $\fsL^\kappa(\Omega_T)$ is bounded in $\fsL^{\kappa'}(\Omega_T)$ with $\frac{1}{\kappa'} = \frac{1}{\kappa} - \frac13$ (provided that $\kappa < 3$), see \cref{thm:smoothing} in Appendix~\ref{sec:smoothing}.

    Using the non-negativity of the $c_i$, we deduce that $i^{k + 1 + \eta}c_i$ is uniformly in $i$ bounded in $\fsL^{q^k_2}(\Omega_T)$. The same holds for $\rho_k$, since $\rho_k = \|i^kc_i\|_{l^1} \leq \|i^{k + 1 + \eta}c_i\|_{l^\infty}\|i^{- (1 + \eta)}\|_{l^1}$ by Hölder's inequality.    
    
    As before, we can iterate this argument and deduce the boundedness of $\rho_k$ in $\fsL^\infty(\Omega_T)$. Indeed, if $\rho_k$ is bounded in $\fsL^{q^k_r}(\Omega_T)$ then, using \eqref{eq:inductionci} we deduce that $i^{k + 1 + \eta}c_i$ is bounded uniformly in $i$ in $\fsL^{q^k_{r+1}}(\Omega_T)$ and hence $\rho_k$ is also bounded in $\fsL^{q^k_{r+1}}(\Omega_T)$. We iterate once more (or twice if we reach exactly $q_r^k = 6$) to deduce the $\fsL^\infty(\Omega_T)$ bound.
    
    Regarding the regularity in Sobolev spaces, we proceed similarly as in the proof of \cref{thm:main2} and we therefore only reproduce the main steps of the argument. First, we see that $(\partial_t - d_i\Delta)i^{k + 1 + \eta}c_i = i^{k + 1 + \eta}(Q_i + F_i)$, where the forcing term $i^{k + 1 + \eta}(Q_i + F_i)$ is bounded in $\fsL^\infty(\Omega_T)$ uniformly in $i$ since it can be controlled by a product of moments $\rho_l$, known to be bounded in $\fsL^\infty(\Omega_T)$ thanks to the previous steps of the proof. Hence, by maximal regularity for all $k\in\N$, $i^{k + 1 + \eta}c_i$ is bounded uniformly in $i$ in any $W^{2, p}(\Omega_T)$ for $1 < p < +\infty$. We can then iterate the argument to deduce that $i^{k + 1 + \eta}(Q_i + F_i)$ is bounded in $W^{2, p}(\Omega_T)$ for $1 < p < +\infty$ and hence $i^{k + 1 + \eta}c_i$ satisfies higher Sobolev estimates.
  \end{proof}

\appendix
  \section{Smoothing estimates for the heat equation in Lebesgue spaces}\label{sec:smoothing}
  Although the following smoothing properties of the heat equation in Lebesgue spaces are well known, we could not find a precise reference in the literature covering the full range of parameters considered here (see, for instance, \cite[Chapter III, Theorem 9.1]{ladyzenskajasolonnikov1968} for a partial result). We therefore include a short proof.
\begin{prop}\label{thm:smoothing}
  Let $\Omega$ be a smooth ($\fsC^2$) bounded domain of $\R^d$, $d\geq 1$. We consider \mbox{$1 < p_f, \, p_u < + \infty$} and $1 \leq q_f, \ q_0 \leq q_u \leq +\infty$ satisfying
  \begin{equation*}
    \begin{cases}
      \frac{1}{p_f} +\frac{d}{2q_f} = 1 + \frac{1}{p_u} + \frac{d}{2q_u},\\
      \frac{d}{2}\left(\frac{1}{q_0}- \frac{1}{q_u}\right) < \frac{1}{p_u}.
    \end{cases}
  \end{equation*}
  Let $T > 0$, $f\in\fsL^{p_f}([0, T], \fsL^{q_f}(\Omega))$ and $u_0\in\fsL^{q_0}(\Omega)$. Let $u$ be the solution to
  \begin{equation*}
    \left\{
    \begin{aligned}
      & \partial_t u - \Delta u = f\quad
      &&\text{in } [0, T]\times\Omega, \\
      & \nabla u\cdot n = 0 \quad
      &&\text{on }  [0,T]\times\partial\Omega, \\
      & u(0,\cdot) = u_0 \quad 
      &&\text{in } \Omega.
  \end{aligned}
  \right.
\end{equation*}

  Then, there exists a constant $C$ depending on $\Omega$, $T$ and $p_u$, $q_u$, $p_f$, $q_f$, $q_0$ such that
  \begin{equation*}
    \|u\|_{\fsL^{p_u}([0, T], \fsL^{q_u}(\Omega))} \leq C \left(\|f\|_{\fsL^{p_f}([0, T], \fsL^{q_f}(\Omega))} + \|u_0\|_{\fsL^{q_0}(\Omega)}\right).
  \end{equation*}
\end{prop}
\begin{proof}
  We recall the smoothing estimates on the semi-group $e^{t\Delta}$ for any \mbox{$1\leq p\leq q \leq +\infty$} (see for instance \cite[Lemma 1.3]{winkler2010a})
  \begin{equation*}
    \|e^{t\Delta}\|_{\fsL^{p}(\Omega)\to \fsL^{q}(\Omega)} \leq C\left(1 + t^{-\frac{d}{2}\left(\frac1p-\frac1q\right)}\right),
  \end{equation*} 
  for some constant $C > 0$ and for any $0 < t < T$.

  Then, using the Duhamel formula $u(t) = e^{t\Delta}u_0 + \int_0^te^{(t - s)\Delta}f(s)\dd s$, we deduce that for some constant $C > 0$
  \begin{align*}
    \|u(t, \cdot)\|_{\fsL^{q_u}(\Omega)}
    &\leq C\left(1 + t^{-\frac{d}{2}\left(\frac{1}{q_0}-\frac1{q_u}\right)}\right)\|u_0\|_{\fsL^{q_0}(\Omega)}\\
    &\phantom{=}+ \int_0^t C \left(1 + (t-s)^{-\frac{d}{2}\left(\frac1{q_f}-\frac1{q_u}\right)}\right)\|f(s, \cdot)\|_{\fsL^{q_f}(\Omega)}\dd s.
  \end{align*}
  By the assumption on $q_0$, we see that the term $1 + t^{-\frac{d}{2}\left(\frac{1}{q_0}-\frac1{q_u}\right)}$ is bounded in $\fsL^{p_u}([0, T])$. For the second term, we first notice that $0 \leq \left(\frac{1}{q_f}-\frac{1}{q_u}\right)\frac{d}{2} < d$ and $\frac{1}{p_f} + \left(\frac1{q_f}-\frac1{q_u}\right)\frac{d}{2} = 1 + \frac{1}{p_u}$. We can hence use Hardy-Littlewood-Sobolev inequality on $[0, T]$ (see for instance \cite[Theorem 4.3]{liebloss2001}) to bound the second term in $\fsL^{p_u}([0, T])$. This yields the conclusion.
\end{proof}

\printbibliography

\end{document}